\documentclass[12pt,twoside,a4paper]{article}

\usepackage{a4,enumerate}
\usepackage[noadjust]{cite}
\usepackage{amssymb,amsmath,amsthm}
\usepackage{comment}
\usepackage{accents,calc}
\usepackage{xcolor}

\swapnumbers
\numberwithin{equation}{section}

\newtheorem{theorem}{Theorem}[section]

\newtheorem{proposition}[theorem]{Proposition}

\theoremstyle{definition}

\newtheorem{remark}[theorem]{Remark}
\newtheorem{remarks}[theorem]{Remarks}

\newcommand\mytitle{On Hausdorff measure and an 
inequality due to Maz'ya} 
\newcommand\lhead{H. Vogt, J. Voigt}
\newcommand\rhead{Hausdorff measure and Maz'ya's inequality}
 \mathchardef\ordinarycolon\mathcode`\:
  \mathcode`\:=\string"8000
  \begingroup \catcode`\:=\active
    \gdef:{\mathrel{\mathop\ordinarycolon}}
  \endgroup

\newcommand\rlim{
\mathchoice{\vcenter{\hbox{${\scriptstyle{+}}$}}}
{\vcenter{\hbox{$\scriptstyle{+}$}}}
{\vcenter{\hbox{$\scriptscriptstyle{+}$}}}
{\vcenter{\hbox{$\scriptscriptstyle{+}$}}}}

\newcommand*\indic{\mathbf{1}}

\newcommand\la{\lambda}
\newcommand*\embed{\hookrightarrow}
\newcommand*\BV{\mathit{BV}}

\newcommand{\spt}{\operatorname{spt}}
\newcommand{\sgn}{\operatorname{sgn}}

\renewcommand{\Re}{\operatorname{Re}}

\newcommand\dist{\operatorname {dist}}

\renewcommand\phi{\varphi}
\newcommand\eps{\varepsilon}
\renewcommand\epsilon{\varepsilon}

\newcommand{\R}{\mathbb{R}\nonscript\hskip.03em}
\newcommand{\N}{\mathbb{N}\nonscript\hskip.03em}

\newcommand{\K}{\mathbb{K}\nonscript\hskip.03em}
\newcommand\cA{\mathcal A}
\newcommand\cB{\mathcal B}

\newcommand\ndash{\rule[.58ex]{\widthof{--}}{0.065ex}}

\newcommand\loc{{\rm loc}}

\newcommand{\kringel}[1]{\accentset{\smash{\raisebox{-0.15ex}{
$\scriptstyle\circ$ } }}{#1}\rule{0pt}{2.2ex}} 

\makeatletter
\let\qedhere@ams\qedhere
\def\qedhere{\@ifnextchar[{\@qedhere}{\qedhere@ams}}
\def\@qedhere[#1]{\tag*{\raisebox{-#1ex}{\qedhere@ams}}}
\makeatletter
\def\env@cases{%
  \let\@ifnextchar\new@ifnextchar
  \left\lbrace
  \def\arraystretch{1.1}%
  \array{@{\,}l@{\quad}l@{}}%
}
\renewcommand\section{\@startsection {section}{1}{\z@}%
                                     {-3.25ex \@plus -1ex \@minus -.2ex}%
                                     {1.5ex \@plus.2ex}%
                                     {\normalfont\large\bfseries}}
\makeatother

\newcommand\restrict{\vphantom f\mskip1mu\vrule\mskip2mu}
\newcommand\set[2]{\bigl\{#1{;}\;#2\bigr\}}
\newcommand\ol{\overline}

\newcommand\hm{\mathcal H}
\newcommand\pd{\partial}
\newcommand\+{\mkern1.5mu}

\newcommand\vol{\operatorname{vol}}
\newcommand\diam{\operatorname{diam}}
\newcommand\rd{\operatorname{rd}}
\newcommand\cC{\mathcal C}
\newcommand\cU{\mkern1mu\mathcal U}
\newcommand\smallbds{\vskip-1\lastskip\vskip5pt plus3pt minus2pt\noindent}

\newcommand\rfrac[2]{\tfrac{#1}{\raisebox{0.1em}{$\scriptstyle#2$}}}
\newcommand\drfrac[2]{\frac{\raisebox{-0.1em}{$#1$}}{\raisebox{0.1em}{$#2$}}}
\newcommand\comp{{\textnormal c}}
\newcommand\Cci{{\displaystyle 
C_{\raise0.2ex\hbox{$\scriptstyle\comp$}}^\infty}}

\renewcommand\le{\leqslant}

\renewcommand\ge{\geqslant}

\newcommand\sse{\subseteq}

\newcommand\di{\mathclose{}\,\mathrm{d}}

\newcommand\abstracttext{\noindent
We give an ``elementary'' proof of an 
inequality due to Maz'ya. 
As a prerequisite we prove an approximation property for the Hausdorff measure. 
We also comment on the relations between Maz'ya's inequality, the isoperimetric 
inequality and the Sobolev inequality.
\vspace{8pt}

\noindent
MSC 2010: 49Q15, 28A75, 46E35
\vspace{2pt}

\noindent
Keywords: Hausdorff measure, isoperimetric inequality, Sobolev 
inequality, functions of bounded variation, perimeter of subsets of $\R^n$
}
\begin{document}
\title{\mytitle}

\author{Hendrik Vogt and J\"urgen Voigt}

\date{}

\maketitle

\begin{abstract}
\abstracttext
\end{abstract}

\section*{Introduction}
\label{intro}

The objective of this note is to present an ``elementary'' proof of the 
following inequality, due to Maz'ya (see \cite[Theorem~4.6.3]{Mazya11}).
If $\Omega\sse\R^n$ is a bounded open 
set, 
$\hm_{n-1}(\pd\Omega)<\infty$, $q=\frac n{n-1}$, and $u\in C(\ol{\Omega})\cap 
W^1_1(\Omega)$, then
\begin{equation}\label{eq-mazya}
\|u\|_{L_q(\Omega)}\le 
c(n)\Bigl(\int_\Omega|{\nabla u(x)}|\di x + 
\int_{\pd\Omega}|{u(x)}|\di\hm_{n-1}(x)\Bigr),
\end{equation}
where $c(n)$ is a constant only depending on the dimension $n$. For the 
Hausdorff measure $\hm_{n-1}$ we refer to Section~\ref{sec-hm}.

In Remark~\ref{rem-rel-inequ}(b) we will comment on the relations between 
\eqref{eq-mazya}, 
the isoperimetric inequality
\begin{equation}\label{eq-isoper}
\vol_n(\Omega)^{\frac{n-1}n}\le c(n)\hm_{n-1}(\pd \Omega),
\end{equation}
valid for any bounded Borel set $\Omega\sse\R^n$, and the Sobolev inequality
\begin{equation}\label{eq-sobolev}
\|u\|_{L_q(\R^n)}\le c(n)\int_{\R^n}|\nabla u(x)|\di x,
\end{equation}
for all $u\in W_1^1(\R^n)$. All three inequalities hold in fact  
with the same optimal constant
\[
c(n)=\frac1{n\omega_n^{1/n}}=\frac{\Gamma(n/2+1)^{1/n}}{n\sqrt\pi}\+,
\]
the isoperimetric constant.
In our proof we will suppose the validity of 
\eqref{eq-sobolev}, with some (non-optimal) constant $c(n)$; we refer to 
\cite[Remark~5.11]{Adams75}, \cite[Th\'eor\`eme~IX.9]{Brezis83} for a proof. 
We will will then derive \eqref{eq-mazya} with a constant that is
larger than the constant in~\eqref{eq-sobolev}.

Rondi \cite[Theorem~2.2]{Rondi17} has shown an inequality more general than 
\eqref{eq-mazya}, where in particular the first term on the right-hand side
is replaced by the variation of~$u$ on~$\Omega$.

Inequalities \eqref{eq-isoper} and \eqref{eq-mazya} are notorious for their 
challenging technical level. The desire to find an accessible proof of 
\eqref{eq-mazya} arose from an application concerning the Dirichlet-to-Neumann 
operator for ``rough domains''; see~\cite{ArendtElst11}. The 
inequality needed in this application is
\begin{equation}\label{eq-mazya-2}
\|u\|_{L_2(\Omega)}^2\le
c(\Omega)\Bigl(\int_\Omega|{\nabla u(x)}|^2\di x + 
\int_{\pd\Omega}|{u(x)}|^2\di\hm_{n-1}(x)\Bigr),
\end{equation}
for all $u\in C(\ol{\Omega})\cap W_2^1(\Omega)$, with a constant $c(\Omega)$ 
only 
depending on the domain, but where the optimality of the constant is 
inessential. In Remark~\ref{rem-rel-inequ}(a) we will indicate how this 
inequality follows from 
\eqref{eq-mazya}. An earlier application of inequality 
\eqref{eq-mazya-2} can be found in~\cite{Daners00}. 

In Section~\ref{sec-hm} we treat an approximation property in the context of 
Hausdorff measures that will be used in the proof of~\eqref{eq-mazya}
in Section~\ref{sec-Maz-ineq}. 
In Section~\ref{sec-bdd-var} we recall the notion of functions of bounded 
variation, and we show that our method also yields an estimate of the total 
variation of $u$ by the right-hand side of~\eqref{eq-mazya}.

\section{On the Hausdorff measure $\hm_d$}
\label{sec-hm}

In this section we present an auxiliary result concerning the 
approximation of the Hausdorff measure of dimension 
$d\in[0,\infty)$ on a metric space $(M,\rho)$. 
We start by introducing some notation and recalling Hausdorff 
measures.

For a set $C\sse M$ we define $\rd(C):=\frac12\diam(C)$, 
and for 
a countable 
collection $\cC$ of subsets of $M$ we define $\rd(\cC):=\sup_{C\in\+\cC}\rd(C)$ 
and 
\[
S_d(\cC):=\omega_d\sum_{C\in\+\cC}\rd(C)^d,
\]
\smallbds
where
\[
 \omega_d
:=\frac{\pi^{d/2}}{\Gamma(\frac{d}{2}+1)},
\]
which is the volume of the unit ball in $\R^d$ if $d\in\N_0$. 
(Even though the notation `$\rd$' should be remindful of `radius', the reader 
should be aware that a set $C$ will not necessarily be contained in a ball with 
radius $\rd(C)$.)

Let $B\sse M$. For $\delta>0$ we put
\[
 \hm_{d,\delta}(B)
:=\inf\set{S_d(\cC)}{ \cC\text{ countable covering of }B,\ 
\rd(\cC)\le\delta}.
\]
\smallbds
Then
\[
 \hm_d^*(B)
:=\lim_{\delta\to 0}\hm_{d,\delta}(B)
=\sup_{\delta> 0}\hm_{d,\delta}(B)\]
is the
{outer $d$-dimensional Hausdorff measure} of $B$.
Carath\'eodory's
construction of measurable sets yields a measure $\hm_d$, the
{$d$-dimensional Hausdorff measure}, and it turns out
that all Borel sets are measurable. If $d\in\N$, and 
$E=\R^d\times\{0\}\sse\R^n$,
then $\hm_d$ is the Lebesgue measure on $\R^d\cong E$. For all of these
properties we refer to \cite[Chap.\,2]{EvansGariepy92} and 
\cite[\protect{2.10.2}]{Federer69}.

Observe that, in the definition of $\hm_{d,\delta}(B)$,
one can also take the infimum over all countable \emph{open} coverings of~$B$
and still obtain the same resulting value for $\hm_d^*(B)$.
Indeed, for $\eps>0$ there exists a 
countable covering $\cC$ with $\rd(\cC)\le\delta/2$ and 
$S_d(\cC)\le\hm_d^*(B)+\eps/2$. Choose $(\eps_C)_{C\in\+\cC}\in(0,\infty)^\cC$ 
such that $\sum_{C\in\+\cC}\eps_C\le\eps/2$. Then for all $C\in\cC$ there 
exists an open set $U_C\supseteq C$ such that
\[
\rd(U_C)\le\delta\quad\text{and}\quad 
\omega_d\rd(U_C)^d\le\omega_d\rd(C)^d+\eps_C.
\]
Then $\cU:=\set{U_C}{C\in\cC}$ is a countable open covering of $B$ 
with 
$\rd(\cU)\le\delta$ and $S_d(\cU)\le\hm_d^*(B)+\eps$.

\sloppy
\begin{proposition}\label{prop-prop-HM}
Let $M$ be a metric space, $d\in[0,\infty)$, and assume that 
$\hm_d(M)<\nobreak\infty$. Let $\eps>0$. 

Then for all $\delta>0$ there exists a countable partition 
$\cA$ of $M$ with $\rd(\cA)\le\delta$, consisting 
of Borel subsets of $M$ and such that
\begin{equation}\label{eq-hm-estimate}
\sum_{C\in\+\cA}\bigl|\hm_d(C)-\omega_d \rd(C)^d\bigr|\le \eps.
\end{equation}
If $M$ is compact, then the partition $\cA$ can be chosen finite.
\end{proposition}

\fussy
\begin{proof}
(i) The definition of $\hm_d$ implies 
that there exists $\delta_\eps>0$ such that for all countable coverings 
$\cC$ of $M$ with $\rd(\cC)\le\delta_\eps$ one has
\begin{equation}\label{eq-hm-lower-est}
\hm_d(M)-\eps\le S_d(\cC).
\end{equation}

(ii) Next we show: for all $\delta>0$ there exists a countable partition 
$\cA$ of $M$ with $\rd(\cA)\le\delta$, consisting 
of Borel subsets of $M$ and such that
\[
S_d(\cA)\le\hm_d(M)+\eps;
\]
if $M$ is compact, then the partition $\cA$ can be chosen finite.

As pointed out above, there exists a countable open covering $\cU$ of $M$ with 
$\rd(\cU)\le\delta$ and $S_d(\cU)\le\hm_d(M)+\eps$. 
If $M$ is compact, then there exists a finite subcovering of $\cU$. A 
standard procedure to produce a disjoint covering yields the desired 
partition~$\cA$.

(iii) Let $0<\delta\le\delta_\eps$ (from step (i)), and let $\cA$ be a 
partition of $M$ as 
in~(ii). 
Let $\cA_1\sse\cA$; then $\cA_1$ is a partition of 
$M_1:=\bigcup\cA_1$. We show that
\begin{equation}\label{eq1--prop-HM}
\hm_d(M_1)-\eps\le S_d(\cA_1)\le\hm_d(M_1)+2\eps.
\end{equation}
Let $\cC_2$ be a countable covering of $M_2:=M\setminus M_1$  with 
$\rd(\cC_2)\le\delta_\eps$. 
Then the covering $\cA_1\cup\cC_2$ of $M$ satisfies
$\rd(\cA_1\cup\cC_2) \le \delta_\eps$, so from \eqref{eq-hm-lower-est} we obtain
\[
\hm_d(M_1)+\hm_d(M_2)-\eps=\hm_d(M)-\eps\le S_d(\cA_1\cup \cC_2)\le 
S_d(\cA_1)+S_d(\cC_2).
\]
Since one can approximate $\hm_d(M_2)$ arbitrarily well by 
$S_d(\cC_2)$, choosing $\cC_2$ suitably, this inequality implies
\[
\hm_d(M_1)-\eps\le S_d(\cA_1),
\]
the left-hand inequality of \eqref{eq1--prop-HM}.

The application of the result obtained so far to the partition 
$\cA_2:=\cA\setminus\cA_1$ of $M_2$ yields
$\hm_d(M_2)-\eps\le S_d(\cA_2)$.
Putting this inequality together with the inequality 
stated in step~(ii) we obtain
\begin{equation*}
S_d(\cA_1) = 
S_d(\cA)-S_d(\cA_2)
\le\hm_d(M)+\eps-(\hm_d(M_2)-\eps)=\hm_d(M_1)+2\eps,
\end{equation*}
the right-hand inequality of \eqref{eq1--prop-HM}.

Now we choose $\cA_1:=\set{C\in\cA}{\omega_d \rd(C)^d\le\hm_d(C)}$ and 
apply \eqref{eq1--prop-HM} with 
$\cA_1$, $M_1$ and with $\cA_2$, 
$M_2$ (as defined above):
\begin{align*}
\hspace{2em} & \hspace{-2em}
\sum_{C\in \cA}\bigl|\hm_d(C)-\omega_d \rd(C)^d\bigr|\\
&=\sum_{C\in\+\cA_1}\bigl(\hm_d(C)-\omega_d \rd(C)^d \bigr)+
\sum_{C\in\+\cA_2}\bigl(\omega_d 
\rd(C)^d-\hm_d(C)\bigr)\\
&=\bigl(\hm_d(M_1)-S_d(\cA_1)\bigr)
+\bigl(S_d(\cA_2)-\hm_d(M_2)\bigr)\le 3\eps.
\qedhere
\end{align*}
\end{proof}

\begin{remark}\label{rem-prop-HM}
The crucial point of the inequality in Proposition~\ref{prop-prop-HM}
is that not only is the sum $S_d(\cA)$ close to $\hm_d(M)$,
but the individual terms $\omega_d\rd(C)^d$ of the sum also approximate
the corresponding terms $\hm_d(C)$, with a small total error.
\end{remark}

\section{Proof of Maz'ya's inequality}
\label{sec-Maz-ineq}

Before entering the proof we recall some properties of distributional 
derivatives.

\begin{remarks}\label{rem-sobolev} Let $\Omega\sse\R^n$ be an open set.

(a) Let $u,v\in L_{1,\loc}(\Omega;\R)$, 
$\nabla 
u,\nabla v\in L_{1,\loc}(\Omega;\R^n)$.
Then $\nabla (u\land v)\in L_{1,\loc}(\Omega;\R^n)$,
\[
|\nabla (u\land v)|\le|\nabla u|+|\nabla v|.
\]
Indeed, from the well-known equality $\partial_ju^+=\indic_{[u>0]}\partial_ju$ 
one easily deduces
\[
\partial_j|u|=\partial_ju^+-\partial_ju^-=\bigl(\indic_{[u>0]}
-\indic_ {[ u\le0]}\bigr)\partial_ju,
\]
for $1\le j\le n$.
Applying this equality to $u\wedge v=\frac12(u+v-|u-v|)$ one obtains
\begin{equation*}
\nabla(u\wedge v) =\indic_{[u\le v]}\nabla u+\indic_{[u>v]}\nabla v,
\end{equation*}
and this equality 
implies the asserted inequality.

(b) Let $u\in L_{1,\loc}(\Omega)$, $1\le j\le n$, 
$\partial_j u\in L_{1,\loc}(\Omega)$. Then
\[
\partial_j|u| = \Re(\ol{\sgn u}\,\partial_j u)\in L_{1,\loc}(\Omega).
\]
Indeed, given $\delta>0$, the equality
\[
\partial_j(|u|^2+\delta)^{1/2}=\partial_j(\ol uu+\delta)^{1/2}
=(|u|^2+\delta)^{-1/2}\Re(\ol u\,\partial_j u)
\]
is straightforward if $u$ is continuously differentiable. Standard 
approximation arguments show that this equality also holds under the present 
hypotheses. 
Letting $\delta\to 0$ one concludes that
\[
\partial_j|u|=\Re\Bigl(\drfrac{\ol u}{| u|}\indic_{[u\ne0]}\,\partial_j 
u\Bigr)
=\Re(\ol{\sgn u}\,\partial_j u).
\]

(c) Let additionally $\Omega\sse\R^n$ be bounded, and let $u\in 
C_0(\Omega)\cap W^1_1(\Omega)$. We show that then $u\in W^1_{1,0}(\Omega)$.
It is sufficient to treat the case when $u$ is real-valued. Splitting 
$u=u^+-u^-$ we reduce the problem to the case $u\ge0$.

For $s>0$ one has $u-s\in W^1_1(\Omega)$, and similarly as in part (a) above 
one obtains 
$\nabla(u-s)^+=\indic_{[u>s]}\nabla(u-s)=\indic_{[u>s]}\nabla u$. 
Note that $u\le s$ on a neighbourhood of $\partial\Omega$
because $u\in C_0(\Omega)$, 
and therefore $\spt(u-s)^+$ is compact, which implies $(u-s)^+\in 
W^1_{1,\comp}(\Omega)$. For $s\to0$ one has $(u-s)^+\to u$, $\nabla(u-s)^+\to 
\indic_{[u>0]}\nabla u$ pointwise and dominated, therefore in $L_1(\Omega)$, 
$L_1(\Omega;\R^n)$, respectively. Because $\indic_{[u>0]}\nabla u=\nabla 
u^+=\nabla u$, this means that $(u-s)^+\to u$ in $W^1_1(\Omega)$, hence 
$u\in\nobreak W_{1,0}^1(\Omega)$.
\end{remarks}

As mentioned in the 
introduction, we will use Sobolev's inequality 
\eqref{eq-sobolev} (where $c(n)$ need not be the optimal constant) in our proof 
of \eqref{eq-mazya}.

\begin{proof}[Proof of Maz'ya's inequality \eqref{eq-mazya}]
Let $u \in C(\ol\Omega)\cap W^1_1(\Omega)$. It follows from 
Remark~\ref{rem-sobolev}(b)
that $|u|\in C(\ol\Omega)\cap W^1_1(\Omega)$ and that 
$\bigl|\nabla |u|\bigr| 
\le|\nabla u|$. This shows that it is sufficient to treat the case 
$u\ge0$.

Let $u\in C(\ol\Omega)\cap W^1_1(\Omega)$, 
$u\ge0$, and let $\eps>0$. Then, by the 
(uniform) continuity of $u$, there exists $\delta\in(0,\eps]$ such that 
$|u(x)-u(y)|<\eps$ whenever $x,y\in\ol\Omega$, $|x-y|<\delta$. 
By
Proposition~\ref{prop-prop-HM} there exists a finite partition $\cA$ of 
$\partial\Omega$ with $\rd(\cA)\le\delta/4$, consisting of Borel sets and 
such that \eqref{eq-hm-estimate} holds. We choose a family 
$(x_C)_{C\in\+\cA}$ with $x_C\in C$ for all $C\in\cA$. 
(Recall that, by 
definition, the sets in a partition are supposed to be 
non-empty.)

Clearly $\set{B_{\R^n}[x_C,\diam(C)]}{C\in\cA}$ is a covering 
of $\partial\Omega$, where we use the notation $B[x,r]$ for the closed ball 
with centre $x$ and radius $r$. For each $C\in\cA$, $s\in(0,\delta/2)$ we 
define a function
\[
\psi_{C,s}(x):=
\begin{cases}
(u(x_C)+\eps)\,\rfrac1s\dist\bigl(x,B(x_C,\diam(C))\bigr) &\text{if }x\in 
B[x_C,\diam(C)+s],\\
\infty &\text{otherwise}.
\end{cases}
\]
Note that $\psi_{C,s}\in W_1^1(B(x_C,\diam(C)+s))$.
Note also that $\diam(C)+s\le\delta$, and hence
$\psi_{C,s}(x)=u(x_C)+\eps>u(x)$ for all $x\in\partial 
B(x_C,\diam(C)+s)\cap\ol\Omega$, $C\in\cA$. 
These properties and Remark~\ref{rem-sobolev}(a) {\ndash} applied repeatedly 
{\ndash} imply that
\[
u_{\eps,s}:=u\land \inf_{C\in\+\cA}\psi_{C,s}\quad\text{on }\ol\Omega
\]
belongs to  $C(\ol\Omega)\cap W^1_1(\Omega)$. 
Since $\delta\le\eps$, the function $u_{\eps,s}$ coincides with $u$
on $\Omega_\eps:=\set{x\in\Omega}{B(x,\eps)\sse\Omega}$,
and as $u_{\eps,s}$ vanishes on 
$\partial\Omega$, Remark~\ref{rem-sobolev}(c) shows that 
$u_{\eps,s}\in W^1_{1,0}(\Omega)$. 

The following computations prepare the application of Sobolev's inequality 
\eqref{eq-sobolev} to $u_{\eps,s}$.
Remark~\ref{rem-sobolev}(a) implies
\begin{equation}\label{eq-mazya-main4}
\int_\Omega|\nabla u_{\eps,s}(x)|\di x \le
\int_\Omega|\nabla u(x)|\di x
+ \sum_{C\in\+\cA}\int_{B(x_C,\diam(C)+s)}|\nabla\psi_{C,s}(x)|\di x.
\end{equation}
Observing that $\psi_{C,s}=0$ on $B(x_C,\diam(C))$ and $|\nabla\psi_{C,s}|=
(u(x_C)+\eps)/s$ on the spherical shell $B(x_C,\diam(C)+s)\setminus 
B(x_C,\diam(C))$, we obtain
\begin{equation}\label{eq-mazya-main5}
\int_{B(x_C,\diam(C)+s)}|\nabla\psi_{C,s}(x)|\di x
=(u(x_C)+\eps)\drfrac1s\omega_n\bigl((\diam(C)+s)^n-\diam(C)^n\bigr).
\end{equation}
We note that, for $s\to0$, the latter expression tends to 
\[
(u(x_C)+\eps)\+n\omega_n\diam(C)^{n-1}=
2^{n-1}\frac{n\omega_n}{\omega_{n-1}}(u(x_C)+\eps)\+\omega_{n-1}\rd(C)^{n-1}.
\]

Recalling that $u\restrict_{\Omega_\eps}=u_{\eps,s}\restrict_{\Omega_\eps}$,  
we conclude from \eqref{eq-sobolev} that
\begin{equation}\label{eq-mazya-main6}
\|u\|_{L_q(\Omega_\eps)}\le\|u_{\eps,s}\|_{L_q(\Omega)}\le 
c(n)\int_\Omega|\nabla u_{\eps,s}(x)|\di x.
\end{equation}
Inserting \eqref{eq-mazya-main4} and \eqref{eq-mazya-main5} into
\eqref{eq-mazya-main6} and taking $s\to0$ we obtain 
\begin{equation}\label{eq-prelim-est}
\|u\|_{L_q(\Omega_\eps)}
\le 
c(n)\Bigl(\+\int_\Omega|\nabla u(x)|\di x
+2^{n-1}\frac{n\omega_n}{\omega_{n-1}}\sum_{C\in\+\cA}(u(x_C)+ 
\eps)\+\omega_{n-1}\rd(C)^{n-1}\Bigr).
\end{equation}
Exploiting~\eqref{eq-hm-estimate} we can estimate the sum on the
right-hand side of~\eqref{eq-prelim-est} by
\begin{align*}
\hspace{2em} & \hspace{-2em}
\sum_{C\in\+\cA}(u(x_C)+\eps)
\bigl(\bigl|\omega_{n-1}\rd(C)^{n-1}-\hm_{n-1}(C)\bigr|+\hm_{n-1}(C)\bigr)\\
&\le(\|u\|_\infty+\eps)\eps +\sum_{C\in\+\cA}\int_C(u(x_C)+\eps)\di\hm_{n-1}\\
&\le(\|u\|_\infty+\eps)\eps +\int_{\pd\Omega}(u+2\eps)\di\hm_{n-1}.
\end{align*}
Because of this inequality, the estimate \eqref{eq-prelim-est} implies
\begin{align*}
\|u\|_{L_q(\Omega_\eps)}
&\le c(n)\biggl(\+\int_\Omega|\nabla u(x)|\di x\\
&\phantom{\le c_(n)\biggl(}
+2^{n-1}\frac{n\omega_n}{\omega_{n-1}}
\Bigl(\+\int_{\partial\Omega}u\di\hm_{n-1}
+\eps(2\hm_{n-1}(\pd\Omega)+\|u\|_\infty+\eps)
\Bigr)\biggr).
\end{align*}
As this inequality holds for all $\eps>0$, we finally obtain
\begin{equation}\label{eq-mazya-main3}
\|u\|_{L_q(\Omega)} \le 
c(n)\Bigl(\+\int_\Omega|\nabla u(x)|\di x
+2^{n-1}\frac{n\omega_n}{\omega_{n-1}}\int_{\partial\Omega}u\di\hm_{n-1}
\Bigr).
\qedhere
\end{equation}
\end{proof}

\begin{remarks}\label{rem-rel-inequ}
(a) Here we show how \eqref{eq-mazya-2} can be derived from \eqref{eq-mazya}.
Using the continuity of the embedding $L_q(\Omega)\embed L_1(\Omega)$ one 
obtains
\begin{equation}\label{eq-mazya-1}
\|u\|_{L_1(\Omega)}\le 
c_1(\Omega)\Bigl(\int_\Omega|{\nabla u(x)}|\di x + 
\int_{\pd\Omega}|{u(x)}|\di\hm_{n-1}(x)\Bigr),
\end{equation}
for all $u\in C(\ol\Omega)\cap W_1^1(\Omega)$. We mention that this inequality 
can also be obtained directly in our proof given above: in the derivation of 
\eqref{eq-mazya-main6} one can use Poincar\'e's 
inequality instead of 
Sobolev's inequality, thereby obtaining an estimate for the $L_1$-norm of $u$.

Let $u\in C(\ol\Omega)\cap W^1_2(\Omega)$. By standard arguments, in particular 
using Remark~\ref{rem-sobolev}(b), we then conclude that 
$\nabla|u|^2=2|u|\,\nabla |u|$, $|u|^2\in W^1_1(\Omega)$. Hence, 
\eqref{eq-mazya-1} implies
\begin{equation}\label{eq1-cor-mazya}
\|u\|_{L_2(\Omega)}^2 \le 
c_1(\Omega)\Bigl(2\int_\Omega|u(x)|\bigl|\nabla|u|(x)\bigr|\di x
+\int_{\partial\Omega}|u|^2\di\hm_{n-1}\Bigr).
\end{equation}
Applying the Cauchy{\ndash}Schwarz inequality
and Young's inequality we get
\begin{align*}
\int_\Omega|u(x)|\bigl|\nabla|u|(x)\bigr|\di x 
&\le \|u\|_{L_2(\Omega)}\Bigl(\int_\Omega|\nabla u(x)|^2\di x\Bigr)^{1/2}\\
&\le \gamma\|u\|_{L_2(\Omega)}^2 +\frac1{4\gamma}\int_\Omega|\nabla u(x)|^2\di x
\end{align*}
for all $\gamma>0$. Inserting this inequality into \eqref{eq-mazya-1}, with 
$\gamma:=\frac1{4c_1(\Omega)}$, and reshuffling terms we finally obtain
\begin{equation*}
\|u\|_{L_2(\Omega)}^2\le 2\+c_1(\Omega)\Bigl(2\+c_1(\Omega)\int_\Omega|\nabla 
u(x)|^2\di x
+\int_{\partial\Omega}|u|^2\di\hm_{n-1}\Bigr).
\end{equation*}

(b)
Here we show how \eqref{eq-isoper} and \eqref{eq-sobolev}, with optimal 
constant, can be derived from~\eqref{eq-mazya}.

If $\Omega\sse\R^n$ is a bounded open set, then applying \eqref{eq-mazya} with 
$u:=\indic_\Omega$ one immediately obtains~\eqref{eq-isoper}. Now let 
$\Omega\sse\R^n$ be a bounded Borel set. We define the bounded open set 
$\Omega_0:=\kringel{\ol\Omega}$ (which may be empty). Then
\[
\pd\Omega_0=\ol{\Omega_0}\setminus\Omega_0\sse
\ol\Omega\setminus\kringel\Omega=\pd\Omega,
\]
and applying \eqref{eq-isoper} for $\Omega_0$ we conclude that
\begin{equation}\label{eq-isoper-Omega_0}
\vol_n(\Omega\cap\Omega_0)^{\frac{n-1}n}\le\vol_n(\Omega_0)^{\frac{n-1}n}
\le c(n)\hm_{n-1}(\pd\Omega_0)\le c(n)\hm_{n-1}(\pd\Omega).
\end{equation}
If $\vol_n(\Omega\setminus\Omega_0)=0$, then \eqref{eq-isoper-Omega_0} shows 
\eqref{eq-isoper} for $\Omega$.
Note that
$\Omega\setminus\Omega_0 \sse\ol\Omega\setminus\kringel\Omega=\pd\Omega$.
Hence, if $\vol_n(\Omega\setminus\Omega_0)>0$, then 
$\hm_n(\pd\Omega)=\vol_n(\pd\Omega)>0$ as 
well, which implies $\hm_{n-1}(\pd\Omega)=\infty$, and 
again~\eqref{eq-isoper} is satisfied for $\Omega$.

Concerning \eqref{eq-sobolev} we first note that \eqref{eq-mazya} clearly 
implies \eqref{eq-sobolev} for all $u\in\Cci(\R^n)$. Applying standard cut-off 
and smoothing procedures one concludes that \eqref{eq-sobolev} holds for all 
$u\in W_1^1(\R^n)$.

(c)
In this part of the remark we sketch a direct proof of \eqref{eq-isoper} for 
bounded open sets $\Omega$ with $C^2$-boundary, with optimal constant.
We first recall the Brunn-Minkowski inequality 
\begin{equation}\label{eq-brunn-mink}
 \la^n(A+B)^{1/n}\ge \la^n(A)^{1/n}+\la^n(B)^{1/n},
\end{equation}
valid for bounded Borel sets $A,B\sse\R^n$, where $\la^n$ denotes the Lebesgue 
measure. We refer to \cite[Theorem~C.7]{Leoni09} for an elementary proof 
of~\eqref{eq-brunn-mink}.
We also recall the Minkowski--Steiner formula 
\[
\hm_{n-1}(\partial\Omega)=\lim_{\eps\to0\rlim}
\rfrac1\eps\bigl(\la^n(\Omega+\eps B)-\la^n(\Omega)\bigr),
\]
where $B$ is the open unit ball in $\R^n$. 
Under our hypothesis of $C^2$-boundary the formula is not too hard to show; the 
essential 
observation for the proof is that, for small $\eps>0$, the set $(\Omega+\eps 
B)\setminus\Omega$ can be written as the disjoint union
\[
\bigcup_{t\in[0,\eps)}\set{x+t\nu(x)}{x\in\pd\Omega},
\]
where $\nu(x)$ denotes the outer unit normal at $x\in\pd\Omega$.

Combining these two ingredients one obtains
\begin{align*}
\hm_{n-1}(\partial\Omega)&=\lim_{\eps\to0\rlim}
\rfrac1\eps\bigl(\la^n(\Omega+\eps B)-\la^n(\Omega)\bigr)\\
&\ge \lim_{\eps\to0\rlim}\rfrac1\eps\bigl(\bigl(\la^n(\Omega)^{1/n}+
\eps\la^n(B)^{1/n}\bigr)^n-\la^n(\Omega)\bigr)\\
&=n\la^n(\Omega)^{(n-1)/n}\la^n(B)^{1/n}
=n\omega_n^{1/n}\la^n(\Omega)^{(n-1)/n},
\end{align*}
for bounded open sets $\Omega\sse\R^n$ with $C^2$-boundary.

On the basis of the above information, a variant of the isoperimetric 
inequality 
(see 
\eqref{eq-isoperimetric-perimeter} below) is proved in 
\cite[Section~9.1.5]{Mazya11}. Using \cite[Proposition~3.62]{AFP00} one then 
concludes \eqref{eq-isoper} for all bounded open $\Omega\sse\R^n$.

(d)
A proof of \eqref{eq-isoper} can also be found in 
\cite[3.2.43, note also 3.2.44]{Federer69}. A proof of 
\eqref{eq-mazya} is given 
in \cite[Example~5.6.2/1 and Theorem~5.6.3]{Mazya11}. 
Another proof of a more general version of
\eqref{eq-mazya} on the basis of \eqref{eq-isoper} and the coarea formula of 
geometric measure theory can be found in \cite{Rondi17}.
\end{remarks}

\section{Functions of bounded variation, the extended Sobolev inequality, and 
Maz'ya's inequality}
\label{sec-bdd-var}

Throughout this section let $\K=\R$.
Let $\Omega\sse\R^n$ be an open set.
A function $u\in L_1(\Omega)$ is of \textit{bounded variation}, 
$u\in\BV(\Omega)$, if the distributional derivatives $\pd_j u$, for 
$j=1,\dots,n$, are finite signed Borel measures. If $u\in\BV(\Omega)$, then the 
\textit{variation} of the $\R^n$-valued vector measure $\pd u$ is given by
\begin{align*}
|\pd u|(A):=\sup\set{\sum_{B\in\cB}|\pd u(B)|}{\cB\text{ finite partition of 
}A\text{ consisting of Borel sets}},
\end{align*}
for Borel subsets $A$ of $\Omega$. There exists a measurable function 
$\sigma\colon\Omega\to S_{n-1}$ such that $\pd 
u=\sigma|\pd u|$ (where $S_{n-1}$ 
denotes the unit sphere in $\R^n$); see \cite[Chap.\,5.1, Theorem 
1]{EvansGariepy92}. The \textit{total variation} of $\pd u$ is given by $|\pd 
u|(\Omega)$.
For more information 
on $\BV$ we refer to \cite[Chap.\,3]{AFP00}, \cite[Chap.\,5]{EvansGariepy92}, 
\cite[Chap.\,13]{Leoni09}.

With these notions we now sketch a proof of the \emph{extended Sobolev 
inequality}
\begin{equation}\label{eq-sobolev-extended}
\|u\|_{L_q(\R^n)}\le c(n)|\pd u|(\R^n)
\end{equation}
for all $u\in\BV(\R^n)$, with the optimal constant $c(n)$ (where $q=\frac 
n{n-1}$).

Recall from Remark~\ref{rem-rel-inequ}(c) the sketch of \eqref{eq-isoper} 
for bounded open $\Omega$ with $C^2$-boundary. Using the coarea formula and 
Sard's lemma one 
can then show \eqref{eq-sobolev} for $\Cci$-functions. 
As in the last paragraph of Remark~\ref{rem-rel-inequ}(b) one extends 
\eqref{eq-sobolev} to all functions $u\in W^1_1(\R^n)$. Finally, let 
$u\in\BV(\R^n)$, and choose a $\delta$-sequence 
$(\rho_k)$ in $\Cci(\R^n)_+$. Then $\rho_k*u\in W^1_1(\R^n)$ for all $k\in\N$, 
$\rho_k*u\to u$ in $L_1(\R^n)$ as $k\to\infty$,
and -- most importantly --  $|\pd(\rho_k*u)|(\R^n)\le |\pd u|(\R^n)$ 
for all $k\in\N$. This last inequality follows from
\[
 |\pd(\rho_k*u)| = |\rho_k*\pd u| \le \rho_k*|\pd u|
\]
and $\int\rho_k(x)\di x=1$.
Now, applying \eqref{eq-sobolev} to $\rho_k*u$ we conclude that
\[
\|\rho_k*u\|_{L_q(\R^n)}\le c(n)\int_{\R^n}|\nabla\rho_k*u(x)|\di x
\le c(n)|\pd u|(\R^n)\quad(k\in\N).
\]
From these inequalities and the convergence of the sequence $(\rho_k*u)$ to $u$ 
in $L_1(\R^n)$ one obtains~\eqref{eq-sobolev-extended}.

\begin{remark}
If $\Omega\sse\R^n$ is a Borel set, 
then one calls $\Omega$ of \textit{finite perimeter} if 
$\indic_\Omega\in\BV(\R^n)$, and 
$P(\Omega):=|\pd \indic_\Omega|(\R^n)$ is called the \textit{perimeter} 
of~$\Omega$. In this case \eqref{eq-sobolev-extended} reads
\begin{equation}\label{eq-isoperimetric-perimeter}
\vol_n(\Omega)^{\frac{n-1}n}=\|\indic_\Omega\|_{L_q(\R^n)}\le c(n)P(\Omega),
\end{equation}
an inequality that sometimes is also called ``isoperimetric inequality''; see 
\cite[Section 9.1.5]{Mazya11}.
\end{remark}

The final issue of this paper is to show that the method of our proof of 
\eqref{eq-mazya} also yields an estimate of the total variation of $\pd u$ by 
the 
right-hand side of \eqref{eq-mazya}, with a non-optimal constant. 

For this purpose we need the following important lower semi-continuity 
property of the 
total variation for $\BV$ functions. Let $(u_k)$ be a sequence in 
$\BV(\Omega)$, $u_k\to u$ in $L_1(\Omega)$, $\sup_{k\in\N}|\pd u_k|(\Omega) 
<\infty$. Then $u\in\BV(\Omega)$, and $|\pd 
u|(\Omega)\le\liminf_{k\to\infty}|\pd 
u_k|(\Omega)$. (See \cite[Chap.\,5.2, Theorem~1]{EvansGariepy92}.)

\begin{proposition}\label{prop-variation}
Let $\Omega$ and $u$ be as in \eqref{eq-mazya}, and 
extend $u$ to $\R^n$ by putting $u\restrict_{\R^n\setminus\Omega}:=0$. Then 
$u\in\BV(\R^n)$,
\begin{equation}\label{eq-inequ-variation}
|\pd u|(\R^n)\le \int_\Omega|\nabla u(x)|\di x + 
2^{n-1}\frac{n\omega_n}{\omega_{n-1}}\int_{\pd\Omega}|u|\di\hm_{n-1}.
\end{equation}
\end{proposition}

\begin{proof}
First we treat the case $u\ge0$.
For $\eps>0$ we start the construction as in the proof of Maz'ya's inequality 
above and obtain $\delta_\eps\in(0,\eps]$ and functions $u_{\eps,s}\in 
W^1_{1,0}(\Omega)$ satisfying \eqref{eq-mazya-main4} for all 
$s\in(0,\delta_\eps/2)$. We also extend $u_{\eps,s}$ to $\R^n$ by zero on 
the complement of $\Omega$ and thereby obtain $u_{\eps,s}\in W^1_1(\R^n)$.
Exploiting the considerations following \eqref{eq-mazya-main4} -- 
but staying with the left hand side of \eqref{eq-mazya-main4} instead of 
applying~\eqref{eq-sobolev} -- we find $s_\eps\in(0,\delta_\eps/2)$ such that
\begin{align*}
\int_{\R^n}|\nabla u_{\eps,s_\eps}(x)&|\di x
\le \int_\Omega|\nabla u(x)|\di x\\
& +2^{n-1}\frac{n\omega_n}{\omega_{n-1}}
\Bigl(\+\int_{\partial\Omega}u\di\hm_{n-1}
+\eps(2\hm_{n-1}(\pd\Omega)+\|u\|_\infty+\eps)
\Bigr)+\eps.
\end{align*}
Clearly $u_{\eps,s_\eps}\to u$ in $L_1(\R^n)$ as $\eps\to0$. Applying the lower 
semi-continuity of the total variation mentioned above we conclude 
\eqref{eq-inequ-variation}.

For general $u\in C(\ol{\Omega})\cap W^1_1(\Omega)$ one then obtains  
\eqref{eq-inequ-variation} by applying $|\pd u|(\R^n) \le |\pd 
u^+|(\R^n) + |\pd u^-|(\R^n)$.
\end{proof}

\begin{remarks}
(a) Combining the inequalities \eqref{eq-sobolev-extended} and 
\eqref{eq-inequ-variation} one again obtains~\eqref{eq-mazya-main3}.

(b) One might hope that, once one has \eqref{eq-inequ-variation} and thereby 
knows that the function $u$ defined in Proposition~\ref{prop-variation} 
belongs to $\BV$,
one can continue and estimate the total variation of $\pd u$ by 
$\int_\Omega|\nabla u(x)|\di x + \int_{\pd\Omega}|u|\di\hm_{n-1}$. Combining 
\eqref{eq-sobolev-extended} with this estimate one would
obtain Maz'ya's inequality \eqref{eq-mazya} with the optimal constant. 
This kind of procedure is successful for the isoperimetric inequality, 
where first it is shown that finite Hausdorff measure of the boundary implies 
finite perimeter, and then fine points of geometric measure theory are applied 
to show that in fact the perimeter is estimated by the Hausdorff measure of the 
boundary. 
\end{remarks}

{
}
\newpage

\noindent
Hendrik Vogt\\
Fachbereich Mathematik\\
Universit\"at Bremen\\
Postfach 330 440\\
28359 Bremen, Germany\\
{\tt 
hendrik.vo\rlap{\textcolor{white}{hugo@egon}}gt@uni-\rlap{\textcolor{white}{%
hannover}}bremen.de}\\[3ex]
J\"urgen Voigt\\
Technische Universit\"at Dresden\\
Fakult\"at Mathematik\\
01062 Dresden, Germany\\
{\tt 
juer\rlap{\textcolor{white}{xxxxx}}gen.vo\rlap{\textcolor{white}{yyyyyyyyyy}}%
igt@tu-dr\rlap{\textcolor{white}{%
zzzzzzzzz}}esden.de}

\end{document}